\documentclass[a4paper,11pt,reqno]{amsart} 
\usepackage{a4,amsbsy,amscd,amssymb,amstext,amsmath,latexsym,oldgerm,enumerate,verbatim,graphicx,xy}

\newtheorem{thm} {Theorem} [section]
\newtheorem{prop}{Proposition} [section]
\newtheorem{lem} {Lemma} [section]
\newtheorem{con} {Conjecture}[section]

\newtheorem{corgl} {Corollary}  

\newtheorem{thmnn}{Theorem}                

\newtheorem{propnn}{Proposition}

\newtheorem{lemnn}{Lemma}

\newtheorem{cornn}{Corollary}

\theoremstyle{definition}

\newtheorem{rem} {Remark} [section]
\newtheorem{rems} [rem]{Remarks}
\newtheorem{exa} [rem] {Example}

\newcommand{\mf}{\mathfrak}
\newcommand{\mc}{\mathcal}
\newcommand{\mb}{\mathbb}

\newcommand{\ov}{\overline}
\newcommand{\un}{\underline}

\newcommand{\sm}{\setminus}         

\newcommand{\ot}{\otimes}           
\newcommand{\la}{\langle}
\newcommand{\ra}{\rangle}

\newcommand{\Hom}{{\rm Hom}}        
\newcommand{\End}{{\rm End}}

\newcommand{\Mat}{{\rm Mat}}

\newcommand{\Sym}{{\rm Sym}} 

\newcommand{\codim}{{\rm codim}}

\newcommand{\tr}{{\rm tr}}
\newcommand{\id}{{\rm id}}

\newcommand{\g}{\mf{g}}
\newcommand{\h}{\mf{h}}

\newcommand{\fb}{{\mf b}}
\newcommand{\fc}{\mf{c}}

\newcommand{\gl}{\mf{gl}}
\newcommand{\spl}{\mf{sl}}

\newcommand{\Dist}{{\rm Dist}}
\newcommand{\GL}{{\rm GL}}

\newcommand{\SL}{{\rm SL}}

\begin{document}

\title{Invariants in divided power algebras}

\begin{abstract}
Let $k$ be an algebraically closed field of characteristic $p>0$, let $G=\GL_n$ be the general linear group over $k$, let $\g=\gl_n$ be its Lie algebra
and let $D_s$ be subalgebra of the divided power algebra of $\g^*$ spanned by the divided power monomials with exponents $<p^s$.
We give a basis for the $G$-invariants in $D_s$ up to degree $n$ and show that these are also the $\g$-invariants.

We define a certain natural \emph{restriction property} and show that it doesn't hold when $s>1$.
If $s=1$, then $D_1$ is isomorphic to the truncated coordinate ring of $\g$  of dimension $p^{\dim(\g)}$ and we conjecture that the restriction property holds and show that this
leads to a conjectural spanning set for the invariants (in all degrees).

We give similar results for the divided power algebras of several matrices and of vectors and covectors, and show that in the second case the restriction property doesn't hold.

We also give the dimensions of the filtration subspaces of degree $\le n$ of the centre of the hyperalgebra of the Frobenius kernel $G_s$.
\end{abstract}

\author[R.~Tange]{Rudolf Tange}
\address
{School of Mathematics,
University of Leeds,
LS2 9JT, Leeds, UK}
\email{R.H.Tange@leeds.ac.uk}

\maketitle

\section*{Introduction}
Let $k$ be an algebraically closed field of characteristic $p>0$, let $G=\GL_n$ be the general linear group over $k$ and let $\g$ be its Lie algebra: the $n\times n$ matrices with entries in $k$. For the representation theory of $G$ and $\g$ it is of interest to understand the centres $U^{[p]}(\g)^\g$ and $\Dist(G)^G$ of the restricted enveloping algebra $U^{[p]}(\g)$ and the hyperalgebra or distribution algebra $\Dist(G)$. In this paper we study their commutative analogues: the truncated symmetric algebra $\ov S(\g)=S(\g)/(x^p\,|\,x\in\g)$ and the divided power algebra $D(\g)$. They are isomorphic to their noncommutative analogues as $G$-modules under the conjugation action.
The connection with the representations of $G$ and $\g$ is described in more detail in Remark~\ref{rems.prelim}.3 and Corollary~3 to Theorem~\ref{thm.invs1} (the hyperalgebras of $G$ and $G_s$) and Remark~\ref{rems.prelim}.4 (the Schur algebra).
To state our results it is more convenient to work with $A_1(\g)=\ov S(\g^*)$ and $D(\g^*)$. This is harmless, since $\g\cong\g^*$ as $G$-modules.

Initially we were interested in describing the invariants for the group and the Lie algebra in $A_1(\g)$ and its higher analogues $A_s(\g)=S(\g^*)/(f^{p^s}\,|\,f\in\g^*)$.
It is easy to see that $A_1(\g)^G$ is bigger than the image of $S(\g^*)^G$ (or $S(\g^*)^\g$): the top degree element (unique up to a scalar multiple) of $A_1(\g)^G$ is not in the image of $S(\g^*)^G$. It turned out to be more convenient to work with the dual versions $D_s(\g^*)$ of the $A_s(\g)$, inside the divided power algebra $D(\g^*)$ where we have the divided power maps. Up to degree $n$ it is easy to give a basis for the invariants in $D(\g^*)$. In fact we can give three different bases, see Section~\ref{ss.invs}. So the task is then to describe the invariants of the subalgebras $D_s(\g^*)$ in terms of these bases. For one of the three aforementioned bases of $D(\g^*)$ we obtain a basis of $D_s(\g^*)$ by forming equivalence class sums for a certain equivalence relation on the basis, for the other two we obtain a basis of $D_s(\g^*)$ by taking a suitable subset of the basis, see Theorem~\ref{thm.invs2}.

We also consider the so-called ``restriction property" for several families of algebras, see Section~\ref{ss.res}. Intuitively, when the restriction property holds one may expect a universal description of the invariants, independent of the rank $n$. When it doesn't hold the description of the invariants will depend on the rank. In all the classical cases (invariants in the coordinate rings of vectors and covectors and of several matrices) the restriction property holds, at least for the group. For the algebras $D(\g^*)$ and $D_s(\g^*)$ that we study, the restriction property almost never holds. We can only conjecture it for $D_1(\g^*)=A_1(\g)$, see Conjecture~\ref{con.surjective_restriction}.


The paper is organised as follows. In Section~\ref{s.prelim} we discuss some, mostly well-known, results about divided power algebras, truncated coordinate rings, polarisation and $\mb Z$-forms, and multilinear invariants of several matrices that we will need later on.

Section~\ref{s.liealg} contains our main result which  describes the $G$-invariants in the algebra $D_s(\g^*)$: Theorem~\ref{thm.invs2}. To prove it, it is more convenient to first work with $D_s(\g)$. Theorem~\ref{thm.invs1} is our main result for this algebra. Infinitesimal invariants are discussed in Proposition~\ref{prop.inf_invs1}.
In Corollary~3 to Theorem~\ref{thm.invs2} we give the dimensions of the filtration subspaces of degree $\le n$ of the centre of the hyperalgebra of $\Dist(G_s)$. In Remark~\ref{rems.res}.3 and 4 we show that the restriction property doesn't hold for the algebras $A_s(\g)$ when $s\ge2$ and also not for the algebras $D_s(\g^*)$ when $s\ge2$. In Section~\ref{ss.dims} we give dimensions for the invariants in the graded pieces of some of the $A_s(\g)$.

In Section~\ref{s.several_matrices} we study the divided power algebra and its ``truncated" subalgebras for several matrices. Theorem~\ref{thm.invs3} describes the $G$-invariants and Proposition~\ref{prop.inf_invs2} describes the infinitesimal invariants. To state and prove these results we first need to state some, mostly well-known, results about conjugacy classes in the symmetric group for a Young subgroup, partial polarisation, and invariants in the full divided power algebra.

In Section~\ref{s.vecs_and_covecs} we study the divided power algebra and its truncated subalgebras for vectors and covectors. Proposition~\ref{prop.invs}(i) describes the $G$-invariants and Proposition~\ref{prop.invs}(ii) describes the infinitesimal invariants.
As preliminaries we first state some, mostly well-known, results about partial polarisation, and orbits in the symmetric group for the multiplication action of a product of two Young subgroups. In Remark~\ref{rems.vecs_and_covecs}.3 we show that the restriction property doesn't hold in this case.

\section{Preliminaries}\label{s.prelim}
Throughout this paper $k$ is an algebraically closed field of characteristic $p>0$ and $s$ is an integer $\ge1$.

\subsection{Lucas's Theorem and Legendre's Theorem}
We remind the reader of two basic results from number theory.
\begin{thmnn}[Lucas's Theorem]
Let $a=\sum_{i\ge0}a_ip^i$ and $b=\sum_{i\ge0}b_ip^i$ be the $p$-adic expansions of the integers $a,b\ge0$.
Then $\binom{a}{b}\equiv\prod_{i\ge0}\binom{a_i}{b_i}$ mod $p$.
\end{thmnn}
\begin{thmnn}[Legendre's Theorem]
Let $\nu_p:\mb Z_{>0}\to\mb Z_{\ge0}$ be the $p$-adic valuation, let $a\ge1$ be an integer, and let $s_p(a)$ be the sum of the $p$-adic digits of $a$.
Then $\nu_p(a!)=\frac{a-s_p(a)}{p-1}$.
\end{thmnn}

\subsection{The divided power algebra and certain subalgebras}\ \\
Let $V=k\ot_{\mb Z} V_{\mb Z}$ be a vector space over $k$ ``defined over $\mb Z$" where $V_{\mb Z}$ has $\mb Z$-basis $(y_1,\ldots,y_m)$. We will denote $1\ot y_i\in V$ just by $y_i$. Inside the symmetric algebra $S(V_{\mb Q})$ of $V_{\mb Q}=\mb Q\ot_{\mb Z}V_{\mb Z}$ we can form the divided power monomials $\prod_{i=1}^my_i^{(t_i)}$ where $t_i\ge0$ and $x^{(t)}=\frac{1}{t!}x^t$. They are linearly independent over $\mb Q$ and their $\mb Z$-span is a $\mb Z$-subalgebra $D(V_{\mb Z})$ of $S(V_{\mb Q})$. Now we put $D(V)=k\ot_{\mb Z}D(V_{\mb Z})$.

The algebra $S(V_{\mb Q})$ has the divided power map $\gamma_i=(x\mapsto x^{(i)}):I_{\mb Q}\to S(V_{\mb Q})$ where $I_{\mb Q}$ consists of the polynomials without constant term. The $\gamma_i$, $i\ge1$, preserve $I_{\mb Z}=D(V_{\mb Z})\cap I_{\mb Q}$ and therefore induce divided power maps $\gamma_i:I_{\mb Z}\to D(V_{\mb Z})$.
These $\gamma_i$ preserve $pI_{\mb Z}$ for $i\ge1$ and, reducing mod $p$ and extending from $\mb F_p=\mb Z/p\mb Z$ to $k$, we obtain divided power maps $\gamma_i=(x\mapsto x^{(i)}):I\to D(V)$, where $I=k\ot I_{\mb Z}$,
which satisfy: 
\begin{enumerate}[{\rm (1)}]
\item $\gamma_0(x)=1$, $\gamma_1(x)=x$ and $\gamma_i(x)\in I$ for $i\ge 1$ and for all $x\in I$,\label{eq.divpower0}
\item $\gamma_i(x+y)=\sum_{j=0}^i\gamma_j(x)\gamma_{i-j}(y)$ for $i\ge 0$ and for $x,y\in I$,\label{eq.divpower_sum} 
\item $\gamma_i(xy)=x^i\gamma_i(y)$ for $i\ge 0$, $x\in D(V)$ and $y\in I$,\label{eq.divpower_prod1}
\item $\gamma_i(x)\gamma_j(x) = \binom{i+j}{i}\gamma_{i+j}(x)$ for $i,j\ge0$ and $x\in I$, and \label{eq.divpower_prod2}
\item $\gamma_i(\gamma_j(x)) = \frac{(ij)!}{(i!)^j j!}\gamma_{ij}(x)$ for $i,j\ge0$ and $x\in I$. \label{eq.divpower_comp}
\end{enumerate}
The algebra $D(V)$ is commutative and graded and has a $\GL(V)$-action which is ``defined over $\mb Z$", so it is clear that the $\gamma_i$ are $\GL(V)$-equivariant.
The span $D_s(V)$ of the (divided power) monomials $\prod_iy_i^{(t_i)}$, $0\le t_i<p^s$ is a $\GL(V)$-stable graded subalgebra of $D(V)$ of dimension $p^{sm}$.
It can be characterised as the distribution algebra or hyperalgebra of the $s$-th Frobenius kernel $V_{a,s}$ of the additive group scheme $V_a$, see \cite[I.4.25]{Jan}.
We denote the graded pieces of degree $r$ of $D(V)$ and $D_s(V)$ by $D^r(V)$ and $D_s^r(V)$.

\begin{lem}\label{lem.divpower}
Let $B=\bigoplus_{r\ge2}D^r_1(V)$. Then $B$ is stable under multiplication and under the $\gamma_i, i\ge1$.
\end{lem}
\begin{proof}
Clearly, $B$ is stable under multiplication. Note that $\binom{\sum_ia_ip^i}{\sum_ib_ip^i}$ is nonzero mod $p$ if $0\le b_i\le a_i<p$ for all $i$ by Lucas's Theorem, and $\frac{p^{i+1}!}{(p^i!)^p p!}$ is nonzero mod $p$ by Legendre's Theorem. So in view of \eqref{eq.divpower_prod2} and \eqref{eq.divpower_comp} it is enough to show that $B$ is stable under $\gamma_p$. Clearly, $B$ is stable under the $\gamma_i$, $1\le i<p$, so by \eqref{eq.divpower_sum} it is enough to show that $\gamma_p(u)=0$ for any divided power monomial $u$ in the $y_i$ of degree $j$ with $2\le j<p$. If $u$ involves at least two variables, then this follows immediately from \eqref{eq.divpower_prod1}. So assume $u=y_i^{(j)}$ for some $i$. Then $\gamma_p(u)= \frac{(jp)!}{(j!)^p p!} y_i^{(jp)}$ by \eqref{eq.divpower_comp}. By Legendre's Theorem the $p$-adic valuation of $\frac{(jp)!}{(j!)^p p!}$ is $\frac{jp-s_p(jp)}{p-1} - (\frac{p(j-s_p(j))}{p-1}+1)$, where $s_p(j)$ denotes the sum of the $p$-adic digits of $j$. Now $s_p(jp)=s_p(j)=j$, since $j<p$, so this $p$-adic valuation equals $j-1$ which is $\ge 1$. So $\frac{(jp)!}{(j!)^p p!}=0 \mod p$.
\end{proof}

\subsection{Truncated coordinate rings}\label{ss.trunc}
Define the ideal $I_s$ of the coordinate ring $A=A(V)=k[V]=S(V^*)$ of $V$ by $I_s=(f^p\,|\,f\in V^*)=({x_i}^{p^s}\,|\,1\le i\le m)$, where $(x_1,\ldots,x_m)$ is the dual basis of $(y_1,\ldots,y_m)$. Put $A_s=A_s(V)=k[V]/I_s$.
We call $A_s(V)$ the \emph{$s$-th truncated coordinate ring of $V$}. It is a commutative graded algebra of dimension $p^{sm}$ and can be characterised as the coordinate ring of the aforementioned infinitesimal group scheme $V_{a,s}$. We denote the graded pieces of degree $r$ of $A(V)$ and $A_s(V)$ by $A^r(V)$ and $A_s^r(V)$.
There is a $\GL(V)$-equivariant isomorphism of graded Hopf algebras $D_s(V)\cong A_s(V)^*$. It maps $\prod_{i=1}^my_i^{(t_i)}$, $0\le t_i<p^s$, to the dual basis element of $\prod_{i=1}^m{x_i}^{t_i}$.
Put $A_s^r=A_s^r(V)$. The top degree of $A_s$ is $N=(p^s-1)m$ and $A_s(N)=k\prod_{i=1}^m{x_i}^{p^s-1}$ is 1-dimensional. Since $\GL(V)$ acts through $\det^{1-p^s}$ on $A_s(N)$, the multiplication defines an $\GL(V)$-invariant pairing $A_s^r\times (A_s^{N-r}\ot\det^{p^s-1})\to k$. This pairing is nondegenerate, so we obtain isomorphisms $(A_s^r)^*\cong A_s^{N-r}\ot\det^{p^s-1}$ and $D_s(V)\cong A_s(V)^*\cong A_s(V)\ot\det^{p^s-1}$ of $\GL(V)$-modules. The latter maps $\prod_{i=1}^my_i^{(t_i)}$ to $\prod_{i=1}^mx_i^{p^s-t_i}$. 

\subsection{The polarisation map and $\mb Z$-forms}\label{ss.polarisation_Zform}
The polarisation map $$P : S^r(V^*)\to (S^rV)^*$$ in degree $r$ sends $f\in S^r(V^*)$ to the
the multi-homogeneous component of degree $(1,\ldots,1)$ of the $r$-variable polynomial function
$(v_1,\ldots,v_r)\mapsto f(v_1+\cdots+v_r)$.
Let $F:V^{\oplus r}\to k$ be $r$-linear, and let $f=(v\mapsto F(v,\ldots,v))\in S^r(V^*)$, then 
$$P(f) = ((v_1,\ldots,v_r)\mapsto \sum_{\sigma\in S_r}F(v_{\sigma(1)},\ldots,v_{\sigma(1)}))\,,$$
where $S_r$ denotes the symmetric group or rank $r$.
We extend $P$ to a linear map from $k[V]=S(V^*)$ to the graded dual $S(V)^{*\rm gr}$ of $S(V)$ and this is an algebra homomorphism. The multiplication on this graded dual comes from the comultiplication on $S(V)$, see \cite[III.11]{Bou}.

Inside $S(V^*_{\mb Q})$ we have the ``divided power $\mb Z$-form" $D(V_{\mb Z}^*)$. The polarisation map over $\mb Q$ maps this $\mb Z$-form onto the standard $\mb Z$-form of the graded dual of $S(V_{\mb Q})$. Reducing mod $p$ we obtain an isomorphism from $D\stackrel{\rm def}{=}D(V^*)$ to $S(V)^{*\rm gr}$. We now identify these two. Then $D^r\stackrel{\rm def}{=}D^r(V^*)=S^r(V)^*=((V^{\otimes r})^*)^{S_r}$: the space of symmetric $r$-linear functions $V^{\oplus r}\to k$, and $D_s^r=D_s^r(V^*)$ consists of the symmetric $r$-linear functions that vanish when $p^s$ arguments are the same. Furthermore the polarisation map over $k$ can now be identified with the map $k[V]=S(V^*)\to D(V^*)$ given by inclusion of $\mb Z$-forms.
We note that for a symmetric $r$-linear function $V^{\oplus r}\to k$ to vanish when $p^s$ arguments are the same it is enough to check that this holds for $r$-tuples of basis vectors. This follows from the fact that the $S_{p^s}$-stabiliser of a nonconstant map $\{1,\ldots,p^s\}\to\{1,\ldots,t\}$, $t$ any integer $\ge2$, is a proper Young subgroup of $S_{p^s}$, so the orbit of such a map has size divisible by $p$. 

We now return to the polarisation map in characteristic $p$. It follows easily from the definition that $P$ has image $D_1$ and kernel $I_1$, so it induces a $\GL(V)$-equivariant isomorphism $A_1=A_1(V)\stackrel{\sim}{\to}D_1(V^*)=D_1$ of graded algebras.

\subsection{Adjoint invariants and symmetric functions}\label{ss.invs}
From now on until the end of Section~\ref{s.liealg} we specialise $V$ to $\g=\gl_n=\End(k^n)$ with $G=\GL_n$ acting by conjugation. So $D=D(\g^*)$ and $A=A(\g)$. The symbol $V$ may now denote another vector space. We work with the bases $(E_{ij})_{1\le i,j\le n}$ of $\g$ with dual basis $(x_{ij})_{1\le i,j\le n}$ of $\g^*$, where $E_{ij}$ is the elementary matrix which is 1 in row $i$ and column $j$ and 0 elsewhere. Note that the trace form on $\g$ is nondegenerate and gives an isomorphism $\g\stackrel{\sim}{\to}\g^*$ of $G$-modules which maps $E_{ij}$ to $x_{ji}$. Note also that the $G$-action factors through the $\SL(\g)$-action, so we have isomorphisms of $G$-modules $D_s^r\cong (A_s^r)^*\cong A_s^{N-r}$, where $N=(p^s-1)n^2$ is the top degree. The $G$-invariants in $D^r=((\g^{\otimes r})^*)^{S_r}$ are the $S_r$-invariants in the space of $G$-invariants of $(\g^{\otimes r})^*$. By ``Schur-Weyl duality" \cite[Sect~4]{DeCProc}, the space of $G$-invariants of $(\g^{\otimes r})^*$ can be described as the image of the group algebra $kS_r$ of the symmetric group $S_r$ under the $S_r$-equivariant linear map
$$\pi\mapsto f_\pi\,,$$
where $f_\pi(X_1,\ldots,X_r)=\prod_{i=1}^r\tr(X_{\sigma_i})$, $\pi=\sigma_1\cdots\sigma_s$ is the disjoint cycle form of $\pi$ (including 1-cycles), $\tr(X_\sigma) \stackrel{\text{def}}{=}\tr(X_{i_1}\cdots X_{i_t})$ for any cycle $\sigma=(i_1,\ldots,i_t)$, and the $S_r$-action on $kS_r$ is by conjugation. This map is injective when $r\le n$. If we work with $\g^{\otimes r}$ instead of the isomorphic module $(\g^{\otimes r})^*$, then the map is given by $\pi\mapsto E_{\pi}$, where $E_\pi=\sum_{i\in\{1,\ldots,n\}^r}\otimes_{l=1}^rE_{i_{\pi(l)}i_l}$.\footnote{Identifying  $\g^{\otimes r}$ with $\End((k^n)^{\otimes r})$, the action of $S_r$ on tensor space is given by $\pi\mapsto E_{\pi^{-1}}$.}

We make some observations about symmetric functions. For the basics we refer to \cite{Mac}. For an integer $i\ge1$ and $X\in\Mat_n$ we define $e_i(X)=\tr(\wedge^iX), h_i(X)=\tr(S^iX)$ and $p_i(X)=\tr(X^i)$. Clearly, the $e_i,h_i,p_i$ can be considered as elements of $k[\g]$ and therefore also as elements of $D(\g)$, see Section~\ref{ss.polarisation_Zform}.
For a partition $\lambda$ of $r$ we define $e_\lambda$ to be the product of the $e_{\lambda_i}$ and we define $h_\lambda$ and $p_\lambda$ in the same way. Via the Chevalley Restriction Theorem (CRT) we can identify these functions with the equally named symmetric functions. Writing $\lambda$ in the form $\lambda=1^{m_1}2^{m_2}\cdots$ we define $z_\lambda=\prod_{i\ge1}i^{m_i}m_i!$ and $u_\lambda=\prod_{i\ge1}m_i!$. Recall that $z_\lambda$ is the order of the centraliser in $S_r$ of a permutation of cycle type $\lambda$. We will call $\frac{1}{z_\lambda}p_\lambda, \frac{1}{u_\lambda}h_\lambda,\frac{1}{u_\lambda}e_\lambda\in S(\g_{\mb Q}^*)$ \emph{divided} $p_\lambda,h_\lambda$ and $e_\lambda$.

For the divided $e_\lambda$'s and $h_\lambda$'s, $\lambda$ a partition of $r$, it is clear that by reduction mod $p$ they can be considered as elements of $(D^r)^G$: $e_\lambda=\prod_{i\ge 1}e_i^{(m_i)}$, $h_\lambda=\prod_{i\ge 1}h_i^{(m_i)}$. We claim that the same is true for the divided $p_\lambda$'s and that, for $n\ge r$, these three families form three bases of $(D^r)^G$. For the first claim we work over $\mb Q$. By \cite[Ex I.6.10, p 110]{Mac} the $\mb Z$-span of the $p_\lambda$'s is the same as that of the $u_\lambda m_\lambda$'s, where the $m_\lambda$'s are the monomial symmetric functions. Taking the ``dual" lattices, i.e. everything that is integral on the lattice via the canonical form, we obtain that the $\mb Z$-span of the divided $p_\lambda$'s is the same as that of the divided $h_\lambda$'s, see \cite[I.4.5, I.4.7]{Mac}. Applying the involution $\omega$ we see that $\mb Z$-span of the divided $p_\lambda$'s is also the same as that of the divided $e_\lambda$'s, see \cite[I.2.6-13]{Mac}. So $\frac{1}{z_\lambda}p_\lambda$ belongs to $D(\g_{\mb Z}^*)$.

To prove the second claim we return to the above $S_r$-equivariant linear map from $kS_r$ onto the $G$-invariant multilinear functions of $r$ matrices. It is injective when $n\ge r$. So in this case $(D^r)^G$ is simply the image of the centre $(kS_r)^{S_r}$ of $kS_r$. If $\pi\in S_r$ has cycle type $\lambda$, then $p_\lambda=(X\mapsto f_\pi(X,\ldots,X))$, so, as an element of $S^r(\g_{\mb Q})^*$ via the polarisation map $P$, it is $\sum_{\sigma\in S_r}f_{\sigma\pi\sigma^{-1}}$. Therefore the sum of the conjugacy class $[\pi]$ is mapped to divided $p_\lambda$. So the divided $p_\lambda$'s form a basis and therefore the divided $e_\lambda$'s and $h_\lambda$'s as well.

\begin{exa}
Take $p=2$. Put $u=\text{divided}\,p_{21}+\text{divided}\,p_{1^3} = \frac{1}{2}p_2p_1+\frac{1}{6}p_1^3\in D(\g_{\mb Z}^*)$. Then $u=(X\mapsto\frac{1}{2}\tr(X^2)\tr(X)+\frac{1}{6}\tr(X)^3)$ corresponds to the symmetric 3-linear function
$$(X,Y,Z) \mapsto \tr(XY)\tr(Z) + \tr(XZ)\tr(Y) + \tr(YZ)\tr(X) + \tr(X)\tr(Y)\tr(Z)\,.$$
In characteristic $p$, this function vanishes when 2 arguments are the same,
so the reduction mod $p$ of $u$ belongs to $D_1$. 
When $n=2$ this function is nonzero (take e.g. $X=E_{12}, Y=E_{21}, Z=E_{11}$), but is zero on triples of diagonal $2\times2$-matrices. The same is true for any symmetric $r$-linear function $r\gl_2\to k$, $r>2$, which vanishes when 2 arguments are the same.
Similarly, $\text{divided}\,p_3 = \frac{1}{3}p_3=((X,Y,Z) \mapsto \tr(XYZ) + \tr(YXZ))$ vanishes in characteristic $2$ when 2 arguments are the same. This function is clearly nonzero for $n\ge2$, but is zero on triples of diagonal matrices for all $n\ge1$.
Note that $e_4=e_2^{(2)}$ on the diagonal matrices for $p=2$ and any $n\ge4$, but not on the $n\times n$ matrices.
\end{exa}

\subsection{The restriction properties}\label{ss.res}
Recall that for a $\g$-module $V$ the subspace of $\g$-invariants in $V$ is defined by $V^\g=\{v\in V\,|\,x\cdot v=0\ \text{for all}\ x\in\g\}$. If $V$ is a commutative $k$-algebra on which $\g$ acts by derivations, for example the differentiated action of an action of $G$ by automorphisms, then any $p$-th power is a $\g$-invariant.

We will occasionally indicate the dependence of our algebras $A_s$ and $D_s$ on the rank $n$ with an extra left subscript $n$.
The embedding $X\mapsto\left(\begin{smallmatrix}X&0\\0&0\end{smallmatrix}\right):\gl_{n-1}\hookrightarrow\gl_n$
induces a $\GL_{n-1}$-equivariant surjections ${}_nA\twoheadrightarrow {}_{n-1}A$ and ${}_nA_s\twoheadrightarrow {}_{n-1}A_s$ and therefore restriction maps
\begin{align}
({}_nA_s)^{\GL_n}&\to ({}_{n-1}A_s)^{\GL_{n-1}}.\label{eq.restriction_group}\\
({}_nA_s)^{\gl_n}&\to ({}_{n-1}A_s)^{\gl_{n-1}}.\label{eq.restriction_liealg}
\end{align}
We say that the algebras $({}_nA_s)_{n\ge1}$ have the \emph{group restriction property} if the above maps \eqref{eq.restriction_group} are surjective for all $n\ge2$.
The \emph{infinitesimal restriction property}, or \emph{Lie algebra restriction property}, can be defined analogously using the maps \eqref{eq.restriction_liealg} and one can define similar restriction maps for the algebras ${}_nA=k[\gl_n]$, ${}_nD$ and ${}_nD_s$.

As is well known, $e_1,\ldots,e_n$ are algebraically independent and generate $A^G$. Clearly, $e_i$ for $\gl_n$ restricts to $e_i$ for $\gl_{n-1}$, so the algebras ${}_nA$, $n\ge1$, have the group restriction property. Furthermore, by Veldkamp's Theorem for $A$, $A$ is generated by $A^p$ and $A^G$, see \cite[Sect~3.5]{PrT} and the references there. So the algebras ${}_nA$, $n\ge1$, also have the infinitesimal restriction property.

\begin{rems}\label{rems.prelim}
1.\ Although the $S_r$-invariants of $kS_r$, i.e. the centre of $kS_r$, in general ($r>n$) does not surject onto the $S_r$-invariants in $((\g^{\otimes r})^*)^G$, it seems that this image does contain the symmetric $G$-invariant multilinear functions of $r$ matrices which vanish when $p$ arguments are the same. The first statement is equivalent to the statement that the algebras ${}_nD$ don't have the restriction property, see Remark~\ref{rems.res}.4. The second statement is implied by Conjecture~\ref{con.surjective_restriction}.\\
2.\ From our discussion of $(D^r)^G$ we get an isomorphism from $(kS_r)^{S_r}$ to the projective limit $\displaystyle\lim_{\stackrel{\longleftarrow}{n}}(S^r(\gl_n)^*)^{\GL_n}$. This map is a characteristic $p$ version the ``characteristic map" from \cite[I.7.3]{Mac}.\\ 
3.\ The map $f\mapsto(X\mapsto f(X-I)):k[\g]\to k[G]$, $I$ the identity matrix, induces $G$-equivariant filtration preserving algebra isomorphisms $A_s\stackrel{\sim}{\to} k[G_s]$, $s\ge1$.
Here the filtrations are given by the powers of the maximal ideals of $0$ resp. $I$. Taking duals we obtain $G$-equivariant filtration preserving coalgebra isomorphisms $\Dist(G_s)\stackrel{\sim}{\to}D_s(\g)\cong D_s$, $s\ge1$, where $\Dist(G_s)$ is the distribution or hyperalgebra of the $s$-th Frobenius kernel $G_s$ of $G$. These fit together to give a $G$-equivariant filtration preserving coalgebra isomorphism
$$\Dist(G)\stackrel{\sim}{\to}D(\g)\cong D \eqno(*)$$ of which the associated graded is a $G$-equivariant isomorphism of Hopf algebras. All this holds in much bigger generality, see \cite[Sect~2]{FP}. We note that the algebra $\Dist(G_1)$ is isomorphic to the restricted enveloping algebra $U^{[p]}(\g)$ of $\g$.

In \cite[Sect 14,15]{OO} Okounkov and Olshanski studied the ``special symmetrisation" map $\sigma:S(\g_{\mb C})\to U(\g_{\mb C})$. It maps the divided power $\mb Z$-form onto the Kostant $\mb Z$-form and after reduction mod $p$ one obtains the inverse of the map (*). Via the Chevalley restriction and Harish Chandra map, the restriction of $\sigma$ to the invariants corresponds to the map $\varphi$ from symmetric functions to ``shifted symmetric functions" which maps the Schur function $s_\lambda$ to the shifted Schur function $s^*_\lambda$. It is not clear to me how to obtain elementary formulas for the images of the symmetric functions $e_\lambda$, $h_\lambda$ and $p_\lambda$ under $\varphi$.\\
4.\ The Schur algebra $S(n,r)$ is isomorphic to $D^r$ as $G\times G$-module, so the centre of $S(n,r)$ is isomorphic to $(D^r)^G$ as vector spaces.
Computer calculations suggest that $(D^r)^G$ has dimension equal to the number of partitions of $r$ of length $\le n$, independent of $p$, and that a spanning set can be obtained by dividing each $h_\lambda$, $\lambda$ a partition of $r$, by the biggest possible integer in the $D(\gl_{n,\mb Z}^*)$ and then reducing mod $p$.\\ 
\end{rems}

\section{The algebras $D_s$ and $D_s(\g)$}\label{s.liealg}
\subsection{Group invariants}
Call a partition \emph{$s$-reduced} if it has $<p^s$ ones. To any partition we can associate an $s$-reduced partition by repeatedly replacing $p^s$ occurrences of 1 by $p^{s-1}$ occurrences of $p$. We will call two partitions \emph{$s$-equivalent} if their associated $s$-reduced partitions are the same. Call two elements of the symmetric group $S_r$ \emph{$s$-equivalent} if their cycle types are $s$-equivalent.
Recall the definition of $E_\pi$, $\pi\in S_r$, from Section~\ref{ss.invs}.

\begin{thm}\label{thm.invs1}
The sums of the $E_\pi$ over the $s$-equivalence classes belong to $D_s(\g)^G$, and when $n\ge r$ they form a basis of $D_s^r(\g)^G$.
\end{thm}
\begin{proof} As we have seen in Section~\ref{ss.invs}, the $E_\pi$ span the $G$-invariants in $\g^{\otimes r}$ and they form a basis when $n\ge r$. So if $n\ge r$, then the sums of the $E_\pi$ over the conjugacy classes form a basis of $D^r(\g)^G=(\g^{\otimes r})^{G\times S_r}$.
The subspace $D_s^r(\g)$ consists of those elements $u$ of $D^r(\g)$ for which $(x_{i_1 j_1} \otimes\cdots\otimes x_{i_r j_r})(u)=0$ for all  $i,j\in\{1,\ldots,n\}^r$ such that $(i_l j_l)_{l\in\{1,\ldots,r\}}$ has at least $p^s$ repetitions. First we observe that $(x_{i_1 j_1} \otimes\cdots\otimes x_{i_r j_r})(E_\pi)=1$ if $j=i\circ\pi$ and 0 otherwise. So, if we put $E_S=\sum_{\sigma\in S}E_\sigma$ for $S\subseteq S_r$, then $(x_{i_1 j_1} \otimes\cdots\otimes x_{i_r j_r})(E_S)=|\{\sigma\in S\,|\,j=i\circ\sigma\}|$ mod $p$. We will now show the following:
\begin{lemnn}
Let $\Lambda\subseteq\{1,\ldots,r\}$ be a set of $p^s$ indices and let $i,j\in\{1,\ldots,n\}^r$ such that $(i_l,j_l)$ is constant for $l\in \Lambda$. We extend the permutations in $\Sym(\Lambda)$ to $\{1,\ldots,r\}$ by letting them fix the elements outside $\Lambda$. Let $\pi\in S_r$.
\begin{enumerate}[{\rm(i)}]
\item If $j\ne i\circ\pi$ or the centraliser $C_{\Sym(\Lambda)}(\pi)$ of $\pi$ in $\Sym(\Lambda)$ does not contain a $p^s$-cycle, then $(x_{i_1 j_1} \otimes\cdots\otimes x_{i_r j_r})(E_{\Sym(\Lambda)\cdot\pi})=0$.
\item If $j = i\circ\pi$ and $C_{\Sym(\Lambda)}(\pi)$ contains a $p^s$-cycle, then $\Lambda$ is $\pi$-stable, and $(x_{i_1 j_1} \otimes\cdots\otimes x_{i_r j_r})(E_{\Sym(\Lambda)\cdot\pi})$ equals\footnote{The reader may want to check that the centraliser of a product of $s$ disjoint $t$-cycles always contains an $st$-cycle.}\\
$\begin{cases}
1 & \text{if $\pi|_\Lambda=\id$,}\\
-1 & \text{if $\pi|_\Lambda$ is a product of $p^{s-1}$ disjoint $p$-cycles, and}\\
0 & \text{otherwise}.
\end{cases}$
\end{enumerate}
\end{lemnn}
\begin{proof}
Let $\Omega$ be the set of permutations $\pi$ with $j=i\circ\pi$. Note that $\Omega$ is $C_{S_r}(i)\times C_{S_r}(j)$-stable, so $\Sym(\Lambda)$ acts on $\Omega$ by conjugation.\\
(i). If $j\ne i\circ\pi$, then $j\ne i\circ\rho$ for all $\rho\in\Sym(\Lambda)\cdot\pi$. Therefore we have $(x_{i_1 j_1} \otimes\cdots\otimes x_{i_r j_r})(E_{\Sym(\Lambda)\cdot\pi})=0$. Now assume that $j=i\circ\pi$. Then
$\Sym(\Lambda)\cdot\pi\subseteq\Omega$ and $(x_{i_1 j_1} \otimes\cdots\otimes x_{i_r j_r})(E_{\Sym(\Lambda)\cdot\pi})=|\Sym(\Lambda)\cdot\pi|$ $\mod p$. So it suffices to show that $\Sym(\Lambda)\cdot\pi$ has size divisible by $p$. Now also assume that $C_{\Sym(\Lambda)}(\pi)$ does not contain a $p^s$-cycle. Then the same holds for $C_{\Sym(\Lambda)}(\rho)$ for all $\rho\in\Sym(\Lambda)\cdot\pi$. Now let $\sigma\in\Sym(\Lambda)$ be any $p^s$-cycle. Then $\la\sigma\ra$ is a $p$-group and all $\la\sigma\ra$-orbits on $\Sym(\Lambda)\cdot\pi$ have size divisible by $p$. So $\Sym(\Lambda)\cdot\pi$ has size divisible by $p$.\\
(ii). Since $j = i\circ\pi$,  we have $(x_{i_1 j_1} \otimes\cdots\otimes x_{i_r j_r})(E_{\Sym(\Lambda)\cdot\pi})=|\Sym(\Lambda)\cdot\pi|$ $\mod p$, as we have seen in the proof of (i). Let $\sigma\in C_{\Sym(\Lambda)}(\pi)$ be a $p^s$-cycle. Then $\Lambda$ is $\pi$-stable, since $\pi$ commutes with $\sigma$. So $\Lambda$ is a union of $\la\pi\ra$-orbits. These orbits are permuted transitively by $\la\sigma\ra$. So they all have the same size, $p^t$ say, $t\in\{0,\ldots,s\}$.

We have $|\Sym(\Lambda)\cdot\pi| = |\Sym(\Lambda)\cdot(\pi|_\Lambda)| = \frac{p^s!}{(p^t)^{p^{s-t}}p^{s-t}!}$, see \cite[I.B.3(1) p171]{Mac}. If we apply the $p$-adic valuation to this we get by Legendre's Theorem
$$\frac{p^s-1}{p-1} - ( tp^{s-t}+\frac{p^{s-t}-1}{p-1})\,.$$
If $t=0$, then $\pi|_\Lambda=\id$ and $|\Sym(\Lambda)\cdot\pi|=1$. Now assume $t=1$. Then $\pi|_\Lambda$ is a product of $p^{s-1}$ disjoint $p$-cycles. Clearly, $|\Sym(\Lambda)\cdot\pi|$ is nonzero mod $p$ (the $p$-adic valuation is zero), so we may assume that $p>2$. For each $a\in\{1,\ldots,p-1\}$ we count how often a $p$-power multiple of a number with remainder $a$ mod $p$ occurs in the list $p^s,p^s-1,\ldots,p^{s-1}+1$ of factors of $\frac{p^s!}{p^{s-1}!}$. It occurs as $a+bp$ for $b=p^{s-2},\ldots,p^{s-1}-1$, as $ap+bp^2$ for $b=p^{s-3},\ldots,p^{s-2}-1,\ldots,$ as $ap^{s-2}+bp^{s-1}$ for $b=1,\dots,p-1$ and finally as $ap^{s-1}$ for $a>1$ and as $p^s$ for $a=1$. That is in total $(p^{s-1}-p^{s-2}) + (p^{s-2}-p^{s-3}) +\cdots+ (p-1) + 1 = p^{s-1}$ times. The product of the nonzero numbers in the prime field is $-1$. So $|\Sym(\Lambda)\cdot\pi| = (-1)^{p^{s-1}}= -1 \mod p$.

Finally assume that $t\ge 2$. Then we have to show that $\frac{p^s-1}{p-1} > tp^{s-t}+\frac{p^{s-t}-1}{p-1}$, i.e. $p^s> tp^{s-t+1}- tp^{s-t}+ p^{s-t}$, i.e. that $p^t>tp-t+1$. This we do by induction on $t$. For $t=2$ this follows from the fact that $p>2-\frac{1}{p}$. Now assume it holds for $t$. Then we have $p\ge 2>1+\frac{1}{p^{t-1}}-\frac{1}{p^t}$. So $p^{t+1} > p^t+p-1 > tp-t+1+p-1 = (t+1)p-(t+1)+1$. So $\Sym(\Lambda)\cdot\pi$ has size divisible by $p$.
\end{proof}

So for $i$,$j$ and $\Lambda$ as in the lemma, the $\Sym(\Lambda)$-orbits $S$ for which the value $(x_{i_1 j_1} \otimes\cdots\otimes x_{i_r j_r})(E_S)$ is nonzero, leave $\Lambda$ stable and come in ``associated pairs": one has cycle structure $1^{p^s}$ on $\Lambda$ and value 1, the other has cycle structure $p^{p^{s-1}}$ on $\Lambda$ and value $-1$. When $T$ is an $s$-equivalence class, then $E_T$ can be written as a sum of certain $E_S$, $S$ a $\Sym(\Lambda)$-orbit and with any such orbit which has nonzero value the associated orbit is also present, so $(x_{i_1 j_1} \otimes\cdots\otimes x_{i_r j_r})(E_T)=0$. It follows that $E_T\in D_s(\g)$.

Now assume that $n\ge r$. Let $\Lambda\subseteq\{1,\ldots,r\}$ be a set of $p^s$ indices, assume $\pi \in S_r$ stabilises $\Lambda$, $\pi|_\Lambda$ is a product of $p^{s-1}$ disjoint $p$-cycles and $\pi'\in S_r$ is the identity on $\Lambda$ and equal to $\pi$ outside $\Lambda$. Denote the $S_r$-conjugacy class of $\sigma\in S_r$ by $[\sigma]$. Note that $[\pi]\ne[\pi']$. Recall from our discussion in Section~\ref{ss.invs} that the $E_{[\sigma]}$ form a basis of $D^r(\g)^G$. To prove the theorem it is enough to show that for any $\Lambda$, $\pi$ and $\pi'$ as above, and any $u\in D_s^r(\g)^G$, $E_{[\pi]}$ and $E_{[\pi']}$ occur with the same coefficient in $u$. Define $i \in\{1,\ldots,n\}^r$ by $i_l=l$ for $l\in\{1,\ldots,r\}\sm \Lambda$ and $i_l=\min(\Lambda)$ for $l\in \Lambda$. Put $j=i\circ\pi=i\circ\pi'$.
By our definition of $i$ and $j$, $j=i\circ\sigma$ implies $\sigma=\pi$ outside $\Lambda$. So the $\Sym(\Lambda)$-orbits of $\pi$ and $\pi'$ form the only associated pair (relative to $i,j$ and $\Lambda$) and the only $\Sym(\Lambda)$-orbit $S$ in $[\pi]$ resp $[\pi']$ for which $E_S$ has nonzero value is that of $\pi$ resp. $\pi'$. So for $u\in (D^r)^G$, written as a linear combination of the $E_{[\sigma]}$, $(x_{i_1 j_1} \otimes\cdots\otimes x_{i_r j_r})(u)$ equals the coefficient of $E_{[\pi']}$ minus the coefficient of $E_{[\pi]}$. This ends the proof of the theorem. 
\end{proof}


\begin{thm}\ \label{thm.invs2}
\begin{enumerate}[{\rm(i)}]
\item The sums of the divided $p_\lambda$'s over the $s$-equivalence classes of the partitions of $r$ belong to $(D_s^r)^G$, and when $n\ge r$ they form a basis of $(D_s^r)^G$.
\item The divided $h_\lambda$'s and the divided $e_\lambda$'s, both with $\lambda=1^{m_1}2^{m_2}\cdots$ such that $m_1<p^s$, belong to $(D_s^r)^G$, and when $n\ge r$ they form two bases of $(D_s^r)^G$.
\end{enumerate}
\end{thm}
\begin{proof}
(i). This is just a reformulation of Theorem~\ref{thm.invs1}, where we now work in the divided power algebra $D$ of $\g^*$ rather than $\g$. As we have seen in Section~\ref{ss.invs} divided $p_\lambda$ corresponds to the sum of the $E_\pi$ over the conjugacy class labelled by $\lambda$.\\ 
(ii). Since these two families are independent, see Section~\ref{ss.invs}, and have the same cardinality as the basis from part (i), it is enough to show that they lie in $D_s$. Recall that $D_s$ is spanned by the divided power monomials in the $x_{ij}$'s with exponents $<p^s$. Both the divided $h_\lambda$'s and the divided $e_\lambda$'s are products of divided powers with exponent $< p^s$ of $e_1=h_1$ and divided powers of elements in the span $B\subseteq D_1$ of the divided power monomials in the $x_{ij}$'s of degree $\ge 2$ and with exponents $<p$.
Using \eqref{eq.divpower_sum} it follows that $\gamma_i(e_1)\in D_s$ for all $i<p^s$. So it is enough to show that $B$ is stable under all divided powers $\gamma_i$, $i\ge 1$. This follows from Lemma~\ref{lem.divpower}.
\end{proof}

\begin{corgl}\ 
The monomials $\prod_{i=1}^ne_i^{(m_i)}$, $m_1<p^s$, belong to $D_s^G$.
Furthermore, for $r\le n$, those with $\sum_{i=1}^nim_i=r$ form a basis of $(D_s^r)^G$.
\end{corgl}
\begin{proof}
This is just a reformulation of the statement about the $e_\lambda$'s in Theorem~\ref{thm.invs2}.
\end{proof}

\begin{rems}\label{rems.invs}
1.\ Let $\ov{A^G}$ denote the image of $A^G$ in $A_1=D_1$. By Veldkamp's Theorem for $k[\g]$, see Section~\ref{ss.res}, $\ov{A^G}$ is also the image of $A^\g$ in $A_1$. Furthermore, by \cite[Thms~8.2 or 8.4]{Pr} it has the monomials in the $e_i$ with exponents $<p$ as a basis. From Corollary~1 it is clear that when $n\ge2p$ the first degree where a ``new" invariant (i.e. not in $\ov{A^G}$) shows up in $A_1$ is $2p$. Indeed $(A_1^{2p})^G$ is the direct sum of the image of $(A^{2p})^G$ and $ke_2^{(p)}$. In the introduction of \cite{T} it is pointed out that $A_1^\g$ modulo $\ov{A^G}$ is isomorphic to $H^1(G_1,I_1)$, where $I_1$ is the ideal from Section~\ref{ss.trunc}. We note that conjecturally $A_1^\g$ and $A_1^G$ are the same, see the remarks after Conjecture~\ref{con.surjective_restriction}.\\
2.\ For $R$ a commutative ring, put $A_{1,R}=R[(x_{ij})_{1\le i,j\le n}]/(x_{ij}^p\,|\,1\le i,j\le n)$. We define $\varphi_p:A_{1,\mb Z}^+\to A_{1,\mb Z}$,  $A_{1,\mb Z}^+$ the truncated polynomials without constant term, by $\varphi_p(u)=\frac{u^p}{p}$. Then $\varphi_p$ descends to a map $\varphi_p:A_{1,\mb F_p}^+\to A_{1,\mb F_p}$. We have $A_{1,\mb F_p}=D_{1,\mb F_p}$, and when $u\in A_{1,\mb F_p}^+$ has no linear or constant term, then $\varphi_p(u)$ can also be computed in the divided power algebra $D_{\mb Z}$ by the same formula. Let $\ov u\in D_{\mb Z}$ be a lift of $u$ (without linear or constant term), let $m\ge0$ be an integer, let $m=\sum_{i=0}^ta_ip^i$ be the $p$-adic expansion of $m$ and write $m! = qp^{\nu_p(m!)}$, where $p$ does not divide $q$. By Legendre's Theorem we have
$\nu_p(m!)=\sum_{i=1}^ta_i\frac{p^i - 1}{p - 1}=\sum_{i=1}^ta_i\nu_p(p^i!)$. So $\ov u^{(m)}=\frac{1}{q}\prod_{i=1}^t(\frac{\ov u^{p^i}}{p^{\nu_p(p^i!)}})^{a_i} =\frac{1}{q}\prod_{i=1}^t(\varphi_p^i(\ov u))^{a_i}$ and therefore
$$u^{(m)}=\frac{1}{q}\prod_{i=1}^t(\varphi_p^i(u))^{a_i}\,.$$
In particular, any divided power monomial $\prod_{i=1}^ne_i^{(m_i)}$ with $m_1<p$ can be expressed as a monomial in $e_1,\ldots,e_n$ together with the iterates of $\varphi_p$ on $e_2,\ldots,e_n$.\\
\end{rems}

\subsection{Infinitesimal invariants}

\begin{lem}\ \label{lem.inf_invs}
Let $V=k^n$ be the natural module for $G$, let $r,t\ge1$ with $n\ge r,t$, and put $W=V^{\oplus r}\oplus(V^*)^{\oplus t}$. For $i\in\{1,\ldots,r\}$ and $j\in\{1,\ldots,t\}$ let $x_i:W\to V$ and $y_j:W\to V^*$ be the $i$-th vector component and $j$-th covector component function and $\la x_i,y_j\ra=((v,w)\mapsto w_j(v_i))\in k[W]^G$ be the bracket function. Then
\begin{enumerate}[{\rm(i)}]
\item the monomials in the $\la x_i,y_j\ra$ with exponents $<p$ form a basis of $k[W]^\g$ over $k[W]^p$, and
\item $((V^{\otimes r}\ot(V^*)^{\otimes t})^*)^\g=((V^{\otimes r}\ot(V^*)^{\otimes t})^*)^G$.
\end{enumerate}
\end{lem}
\begin{proof}
(i).\ We will verify the hypotheses of \cite[Thm 5.5]{Skr}. Using the notation in \cite{Skr} we have that $\fc_\g(W)=\dim\g-\min_{x\in W} \dim\g_x=n^2-(n-r)(n-t)=(r+t)n-rt$, since $n\ge r,t$,
and $\dim(W)-\fc_\g(W)=rt$. Let $U\subseteq W$ be the set of points $(v,w)\in W$ where the differentials $d_{(v,w)}\la x_i,y_j\ra$ are linearly dependent. We have $d_{(v,w)}\la x_i,y_j\ra = ((z,u) \mapsto\la v_i,u_j\ra+\la z_i,w_j\ra) = f_j(v_i)+g_i(w_j)\in W^*=(V^*)^{\oplus r}\oplus V^{\oplus t}$, where $f_j$ embeds $V$ in the $(r+j)$-th position in $W^*$ and $g_i$ embeds $V^*$ in the $i$-th position of $W^*$. It is now easy to check that the differentials of the $\la x_i,y_j\ra$ at $(v,w)$ will be independent if $v\in V^{\oplus r}$ is independent or if $w\in (V^*)^{\oplus t}$ is independent. Since $n\ge r,t$ we can indeed choose $v$ and $w$ like this, so we obtain that $\codim(W\sm U)\ge 2$.\\ 
(ii).\ This follows from (i), since $((V^{\otimes r}\ot(V^*)^{\otimes t})^*)^\g$ consists of the multilinear functions in $k[W]^\g$, so the $p$-th powers cannot be involved. 
\end{proof}

\begin{prop}\ \label{prop.inf_invs1}
Assume $r\le n$ and put $N=(p^s-1)n^2$. Then $(D^r)^\g=(D^r)^G$ and $(A_s^{N-r})^\g=(A_s^{N-r})^G$ for $r\le n$.
\end{prop}
\begin{proof}
Since $A_s^{N-r}\cong D_s^r$ as $G$-modules and $D_s^r$ is a $G$-submodule of $D^r$, it is enough to prove the first assertion. Put $V=k^n$. Since $D_s^r\subseteq D^r\subseteq(\g^{\otimes r})^*\cong (V^{\otimes r}\ot(V^*)^{\otimes r})^*$ it is enough to show that $((V^{\otimes r}\ot(V^*)^{\otimes r})^*)^\g$ equals \break $((V^{\otimes r}\ot(V^*)^{\otimes r})^*)^G$ for $r\le n$ which follows from Lemma~\ref{lem.inf_invs}(ii).
\end{proof}

One can form the divided power algebra of a vector space $V=k\ot_{\mb Z}V_{\mb Z}$ where $V_{\mb Z}$ is any free $\mb Z$-module. If $(x_i)_{i\in I}$ is a basis of $V_{\mb Z}$ one just has to work with monomials $\prod_{i\in I}x_i^{(m_i)}$ with all but finitely many $m_i$ zero. For a family of variables $(x_i)_{i\in I}$ we put $D((x_i)_{i\in I})=D(k\ot_{\mb Z}V_{\mb Z})$ and $D_s((x_i)_{i\in I})=D_s(k\ot_{\mb Z}V_{\mb Z})$ where $V_\mb Z$ is the free $\mb Z$-module on $(x_i)_{i\in I}$.

\begin{corgl}[to Theorem~\ref{thm.invs2}]
$$\lim_{\stackrel{\longleftarrow}{n}}({}_nD_s)^{\gl_n}=\lim_{\stackrel{\longleftarrow}{n}}({}_nD_s)^{\GL_n}=D_s(e_1)\ot D((e_i)_{i\ge2}),$$ where $D((e_i)_{i\ge2})$ is graded such that $e_i^{(m)}$ has degree $mi$, and the limit is in the category of graded $k$-algebras.
\end{corgl}
\begin{proof}
This follows from Proposition~\ref{prop.inf_invs1} and Corollary~1 to Theorem~\ref{thm.invs2}.
\end{proof}

\begin{corgl}[to Theorem~\ref{thm.invs2}]
Denote the centre of $\Dist(G_s)$ by $Z_s$ and for a subspace $W$ of $\Dist(G_s)$ denote by $F^rW$ the intersection of $W$ with the $r$-th filtration subspace of $\Dist(G_s)$. Assume that $r\le n$.
\begin{enumerate}[{\rm(i)}]
\item $F^rZ_s=F^r\Dist(G_s)^G=F^r\Dist(G_s)^\g$.
\item The dimension of $F^rZ_s$ is the number of partitions of $0,1,\ldots,r$ with $<p^s$ ones.
\end{enumerate}
\end{corgl}
\begin{proof}
This follows from Remark~\ref{rems.prelim}.3, Proposition~\ref{prop.inf_invs1} and Theorem~\ref{thm.invs1}.
\end{proof}

\begin{rems}\label{rems.inf_invs}
1.\ The referee mentioned to me the following generalisation of Lemma~\ref{lem.inf_invs}. Call a polynomial dominant weight, i.e. a partition of length $\le n$, \emph{$p^s$-restricted} if $\lambda_i-\lambda_{i+1}<p^s$ for $i=1,\ldots,n-1$, and $\lambda_n<p^s$. Furthermore, call a semisimple $G$-module \emph{$p^s$-restricted} if all its irreducible submodules have \emph{$p^s$-restricted} highest weight. Then we have the following result.
\begin{propnn}
Let $M$ and $N$ be finite dimensional polynomial $G$-modules, homogeneous of degrees $r$ and $t$. If $r\le t$ and $N$ has $p^s$-restricted socle or if $r\ge t$ and $M$ has $p^s$-restricted head, then $\Hom_G(M,N)=\Hom_{G_s}(M,N)$.
\end{propnn}
This result is not hard to prove using standard facts about polynomial modules, see \cite[App A.1-3]{Jan}, contravariant duality, see \cite[II.1.16,2.12,2.13]{Jan} or \cite[2.7,5.4c]{Green}, and the arguments from \cite[II.3.16]{Jan}: One first reduces to the first alternative using contravariant duality, then one reduces to the case that $N$ is an injective indecomposable in the polynomial category, and then one proves the assertion by induction on the number of composition factors, where the assumption $r\le t$ is needed for the basis case that $M$ is irreducible.

From the above result one easily deduces Lemma~\ref{lem.inf_invs}. Indeed for $n\ge r$ we have $\Hom_G(V^{\otimes r},X)\cong X_\omega$ where $X_\omega$ is the $\omega$-weight space and $\omega=(1,\ldots,1,0,\ldots,0)$ ($r$ ones), see \cite[A.22,23]{Jan} or\cite[6.2g Rem~1, 6.4f, 6.4b]{Green}.
So $V^{\otimes r}$ has $p$-restricted head, and, by contravariant duality, $p$-restricted socle.\\
2.\ The conclusion of Lemma~\ref{lem.inf_invs} does not hold when $n<r$ or $n<t$. For example, if we have $n=1, r\ge2,t\ge1$ then $x_1^hx_2^{p-h}$, $1<h<p$ is a $\g$-invariant, but it doesn't belong to the $k[X]^p$-algebra generated by the $x_iy_j$.\\ 
3.\ I checked with the computer that $\dim (D^r)^\g >\dim (D^r)^G$ when $p=2, n=2$ or $n=3$, and $r=n+1$. In the first case I got $8>5$, in the second case $31>23$. When $p=3, n=2$, and $r=5$ I got $45>42$.

For $p = 2, n = 2, r = 3$ one can easily describe a $\g$-invariant in $D^r(\g)=(\g_n^{\otimes r})^{S_r}$ which is not a $G$-invariant. One can take the sum of the 3 $S_3$-conjugates of $(E_{11}+E_{22})\ot E_{12}\ot E_{12}$, i.e. $(E_{11}+E_{22})E_{12}^{(2)}$.\\
4.\ Take $n=2$. Let $H$ be the group of diagonal matrices in $G$ and let $\h$ be its Lie algebra. It is easy to check that the nonzero $H$-weights in $A_1$ are also nonzero for $\h$.
So the $H$-action on $A_1^\g$ is trivial. Of course the same holds for all $G$-conjugates of $H$. From the density of the semisimple elements in $H$ it now follows that $A_1^\g=A_1^G$.
This argument was mentioned to me by S.~Donkin. It is not difficult to show that $\dim(A_1^G)=\frac{3p^2-p}{2}$ and that $e_1=\tr, e_2=\det$ and $e_2^{(2)}$ generate $A_1^G$ by reducing to the $\spl_2$-case when $p>2$. 
\end{rems}

\subsection{The restriction property}
Recall that there is a $G$-equivariant isomorphism $D_1\cong A_1$ of graded algebras.

\begin{con}\label{con.surjective_restriction}
The algebras $({}_nA_1)_{n\ge1}$ have the infinitesimal restriction property.
\end{con}
If this conjecture holds, then $A_1^\g=A_1^G$ by Proposition~\ref{prop.inf_invs1} and the monomials $\prod_{i=1}^ne_i^{(m_i)}$, $m_1<p$, span $A_1^\g$ by Corollary~1 to Theorem~\ref{thm.invs2}.
The point is that the restriction property allows us to reduce to the situation that $n$ is $\ge$ the degree $r$.
Conversely, if these monomials span $A_1^\g$, then $A_1^\g=A_1^G$ and the algebras $({}_nA_1)_{n\ge1}$ have the infinitesimal restriction property.
Note that by Remark~\ref{rems.prelim}.3 $A_1^\g=A_1^G$ implies that the centre $U^{[p]}(\g)^\g$ of $U^{[p]}(\g)$ is contained in the centre $\Dist(G)^G$ of $\Dist(G)$, see \cite[Lem~6.5]{Hab}.

\begin{rems}\label{rems.res}
1.\ We consider the surjectivity of the map $({}_NA_1)^{\GL_N}\to ({}_nA_1)^{\GL_n}$, $N>n$. By Remark~\ref{rems.inf_invs}.4 it is surjective for $n=2$, since the generators there lift to any $({}_NA_1)^{\GL_N}$. I also checked that it is surjective for $n=3$ and $p=2,3,5$, $n=4$ and $p=2$, $3$ (up to degree 8), $n=5$ and $p=2$ (up to degree 7), $p=3$ (up to degree 6).
This was done by checking in each of these cases that the monomials from Corollary~1 to Theorem~\ref{thm.invs2}, span $({}_nA_1^r)^{\GL_n}=({}_nD_1^r)^{\GL_n}$.\\ 
2.\ We consider the conjecture $A_1^\g=A_1^G$. By Remark~\ref{rems.inf_invs}.4 it holds for $n=2$. I checked it with the computer for $n=3$ and $p=2,3,5$, $n=4$ and $p=2$, $3$ (up to degree 7) and $5$ (up to degree 6), $n=5$ and $p=2$ (up to degree 5), $3$ (up to degree 5).
\\
3.\ The algebras $({}_nA_s)_{n\ge1}$, $s\ge 2$, don't have the group or Lie algebra restriction property. I checked this for the restriction ${}_3A_2^{10}\to{}_2A_2^{10}$ when $p=2$: $({}_2A_2^{10})^{\GL_2}$ is spanned by $x_{11}^2 x_{12}^3 x_{21}^3 x_{22}^2+x_{11}^3 x_{12}^2 x_{21}^2 x_{22}^3$ and $x_{11}^3 x_{12}^3 x_{21}^3 x_{22}+x_{11}^2 x_{12}^3 x_{21}^3 x_{22}^2+x_{11} x_{12}^3 x_{21}^3 x_{22}^3$, but the image in ${}_2A_2$ of $({}_3A_2^{10})^{\fb_3}$, $\fb_3$ the upper triangular matrices in $\gl_3$, is spanned by the first element.\\ 
4.\ The algebras $({}_nD_s)_{n\ge1}$, $s\ge2$, and $({}_nD)_{n\ge1}$ don't have the group or Lie algebra restriction property. By Proposition~\ref{prop.inf_invs1} and Theorem~\ref{thm.invs2}(i) it is enough to check that the dimension of the span of the sums of the divided $p_\lambda$'s is $<\dim({}_nD_s^r)^G$. First we consider the case $n=2$.
For $r=5$, $p=2$ I got $1<2$ for $s=2$ and $2<3$ for $s\ge3$,
for $r=8$, $p=3$ I got $4<5$ for $s\ge2$, and
for $r=14$, $p=5$ I got $7<8$ for $s\ge2$.
In the case $n=3$, $r=6$, $p=2$ I got $4<5$ for $s=2$ and $6<7$ for $s\ge3$.
\end{rems}

\subsection{Dimensions of some of the $A^r_s$}\label{ss.dims}
We give some dimensions that we calculated using a computer program. For $n=2$ the dimensions of the $A^r_s$ were always given as the coefficients of the polynomial $\frac{1-T^{p^s}}{1-T}\times\frac{1-T^{3(p^s-1)+2}}{1-T^2}\in\mb Z[T]$ which we calculate as $\frac{1-T^{p^s}}{1-T^2}\times\frac{1-T^{3(p^s-1)+2}}{1-T}$ for $p=2$.
The total dimension was always $p^{2s}+\frac{p^s(p^s-1)}{2}$. 
We checked the cases $s=2, p=2,3,5$ and $s=3,p=2$. For the case $s=1$, see Remark~\ref{rems.inf_invs}.4.

In the table below we give dimensions for $n\ge3$. Let $\ov{A^G}$ denote the image of $A^G$ in $A_s$. The first row gives the dimensions of the $(A^r_s)^G$, the second row gives the dimensions of the graded pieces of $\ov{A^G}$, and the third row, if it exists, gives the dimensions of the $(A^r_s)^\g$.
If the dimensions can be computed in all degrees, then the single number to the right gives the total dimension.

\begin{figure*}[htbp]
    \centering 
        \hspace*{-.5cm}\includegraphics[clip, trim=1.7cm 18.5cm 0.5cm 0.5cm, width=1.50\textwidth]{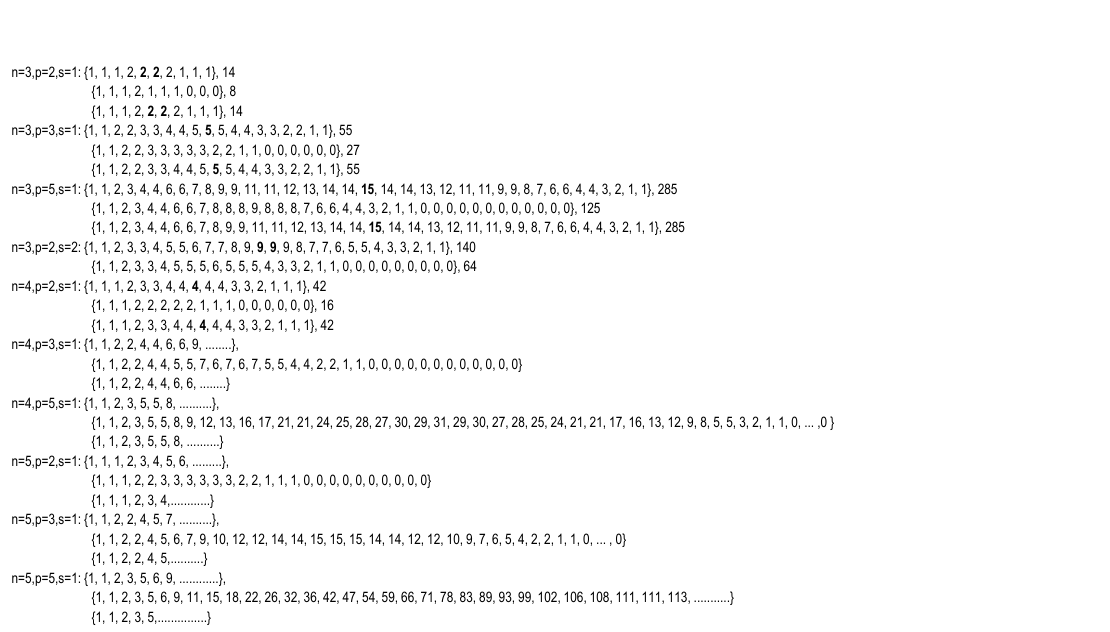}
    Dimensions of the invariants in some of the $A_s^r$
\end{figure*}

\newpage
\section{Several matrices} \label{s.several_matrices}
\newcommand{\bs}{\boldsymbol}
In this section we study the invariants in the algebras $D_s((\g^{\oplus m})^*)$.

\subsection{Conjugacy classes for the conjugation action of $S_\alpha$ on $S_r$}\label{ss.Young_conjugacy}
We recall some notation and results from \cite{Don} about conjugacy classes of a Young subgroup in $S_r$. For a finite sequence $\un i=(i_1,\ldots,i_t)$ of elements of $\{1,\ldots,m\}$ we define ${\rm Content}(\un i)$ to be the $m$-tuple whose $j$-th component is the number of occurrences of $j$ in $\un i$. We say that sequences $\un i$ and $\un j$ as above are equivalent if one is a cyclic shift of the other, we denote the equivalence class of $\un i$ by $[\un i]$ and we put $|[\un i]|=t$. We will call these equivalence classes \emph{cycle patterns}. Clearly, equivalent sequences have the same content, so the content function is also defined on cycle patterns. For $l\ge1$ we define the $l$-th power of $\un i$ by
$$[\un i]^l=[\underbrace{i_1,\ldots,i_t,\ldots,i_1,\ldots,i_t}_{\text{$l$ copies of $\un i$}}]\,.$$
We call a cycle pattern \emph{primitive} if it is not the $l$-th power of another cycle pattern for some $l\ge2$ and we denote the set of primitive cycle patterns by $\Phi$. Let $\mc P$ be the set of partitions. For $\lambda=(\lambda_1,\lambda_2,\ldots)\in\mc P$ we put $|\lambda|=\sum_{i\ge1}\lambda_i$ and we denote the length of $\lambda$, i.e. the number of nonzero parts of $\lambda$, by $l(\lambda)$. For a function $\bs\lambda:\Phi\to\mc P$ such that all but finitely many values are the empty partition we define the \emph{content} of $\bs\lambda$ to be $\sum_{b\in\Phi}|\bs\lambda(b)|{\rm Content}(b)$ and we denote the set of such functions with content $\alpha$ by $\Theta_\alpha$.

Now fix a composition $\alpha=(\alpha_1,\ldots,\alpha_m)$ of $r$. For $i\in\{1,\ldots,m\}$ put $\Delta_i=\{j\in\mb Z\,|\,\sum_{l=1}^{i-1}\alpha_l<j\le\sum_{l=1}^i\alpha_l\}$. Define $\zeta:\{1,\ldots,r\}\to\{1,\ldots,m\}$ by $\zeta(j)=i$ when $j\in\Delta_i$. Let $S_\alpha$ be the simultaneous stabiliser of the $\Delta_i$ in $S_r$. Note that $S_\alpha\cong S_{\alpha_1}\times\cdots\times S_{\alpha_m}$.
For a cycle $\sigma=(i_1,\ldots,i_t)\in S_r$ we put $[\sigma]=[\zeta(i_1),\ldots,\zeta(i_t)]$. We can associate to every $\pi$ with disjoint cycle decomposition $\pi=\prod_{j\in J}\sigma_j$ the multiset of cycle patterns $\la[\sigma_j]\,|\,j\in J\ra$. This multiset is equal to $\la b^{\bs\lambda(b)_i}\,|\,b\in\Phi,1\le i\le l(\bs\lambda(b))\ra$ for a unique $\bs\lambda\in\Theta_\alpha$ which we call the \emph{$S_\alpha$ cycle type} of $\pi$. Clearly, $\pi,\pi'\in S_r$ are $S_\alpha$-conjugate if and only if they have the same $S_\alpha$ cycle type.

\subsection{Partial polarisation}\label{ss.partial_polarisation}
Let $\alpha$, $r$, the $\Delta_i$, $S_\alpha$ and $\zeta$ be as in the previous section and let $V$ be a vector space over $k$. The algebra $S(V^{\oplus m})=S(V)^{\ot m}$ is $\mb Z^m$-graded and we denote the piece of degree $\alpha$ by $S^\alpha(V^{\oplus m})$. We apply analogous notation to the algebras $S((V^{\oplus m})^*)$, $D(V^{\oplus m})$ and $D_s(V^{\oplus m})$. Note that $S^\alpha(V^{\oplus m})\cong S^{\alpha_1}(V)\otimes\cdots\otimes S^{\alpha_m}(V)$, so $S^\alpha(V^{\oplus m})^*$ can be regarded as the $r$-linear functions $V^{\oplus r}\to k$ which are symmetric in each of the sets of positions $\Delta_i$, i.e. which are $S_\alpha$-invariants. For an integer $t\ge0$ let $\un 1_t$ denotes the all-one vector of length $t$.
The partial polarisation map $P_\alpha:S^\alpha((V^{\oplus m})^*)\to S^\alpha(V^{\oplus m})^*$ sends $f\in S^\alpha((V^{\oplus m})^*)$ to
the multi-homogeneous component of degree $(\un1_{\alpha_1},\ldots,\un1_{\alpha_m})$
of the $r$-variable polynomial function
$$(v^1_1,\ldots,v^1_{\alpha_1},\ldots,v^m_1,\ldots,v^m_{\alpha_m})\mapsto f(v^1_1+\cdots+v^1_{\alpha_1},\ldots,v^m_1+\cdots+v^m_{\alpha_m})\,.$$
If $F:V^{\oplus r}\to k$ is $r$-linear and
$f=((v_1,\ldots,v_m)\mapsto F(v_{\zeta(1)},\ldots,v_{\zeta(r)}))$, then
$$P_\alpha(f)=((v_1,\ldots,v_r)\mapsto\sum_{\sigma\in S_\alpha}F(v_{\sigma(1)},\ldots,v_{\sigma(r)}))\,.$$

As in Section~\ref{ss.polarisation_Zform} we obtain isomorphisms $D^\alpha((V^{\oplus m})^*)\cong S^{\alpha}(V^{\oplus m})^*$. Under these isomorphisms $D_s^\alpha((V^{\oplus m})^*)$ can be regarded as the $r$-linear functions $V^{\oplus r}\to k$ which are symmetric in each of the sets of positions $\Delta_i$ and which vanish when the arguments in $p^s$ positions within a $\Delta_i$ are the same. Furthermore, these isomorphisms are compatible with the isomorphism $D((V^{\oplus m})^*)\cong S(V^{\oplus m})^{*\rm gr}$ from Section~\ref{ss.polarisation_Zform}.

\subsection{Invariants in the algebra $D((\g^{\oplus m})^*)$}
We keep the notation of Section~\ref{ss.Young_conjugacy}.
For $f\in k[\g]^G$ and $b=[i_1,\ldots,i_t]$ a cycle pattern define $f_b\in k[\g^{\oplus m}]^G$ by
$$f_b(x_1,\ldots,x_m)=f(x_{i_1}\cdots x_{i_t})\,.$$
For $\bs\lambda\in\Theta_\alpha$ define $p_{\bs\lambda}=\prod_{b\in\Phi}p_{\bs\lambda(b),b}$, $e_{\bs\lambda}=\prod_{b\in\Phi}e_{\bs\lambda(b),b}$ and $h_{\bs\lambda}=\prod_{b\in\Phi}h_{\bs\lambda(b),b}$.
Furthermore define $u_{\bs\lambda}=\prod_{b\in\Phi}u_{\lambda(b)}$ and $z_{\bs\lambda}=\prod_{b\in\Phi}z_{\lambda(b)}$, and call
$\frac{1}{z_{\bs\lambda}}p_{\bs\lambda}$, $\frac{1}{u_{\bs\lambda}}h_{\bs\lambda}$, $\frac{1}{u_{\bs\lambda}}e_{\bs\lambda}\in S((\g_{\mb Q}^{\oplus m})^*)$ \emph{divided} $p_{\bs\lambda}$, $h_{\bs\lambda}$ and $e_{\bs\lambda}$.
As shown in \cite{Don} $z_{\bs\lambda}$ is the order of the centraliser in $S_\alpha$ of an element in $S_r$ of $S_\alpha$ cycle type $\lambda$.
Clearly, the divided $h_{\bs\lambda}$ and $e_{\bs\lambda}$ can be considered as elements of $D^\alpha((\g^{\oplus m})^*)^G$ by reduction mod $p$. We will now show that the same holds for the divided $p_{\bs\lambda}$ and that, for $n\ge r$, they form three bases of $D^\alpha((\g^{\oplus m})^*)^G$.
First we note that for $b\in\Phi$ the map $f\mapsto f_b$ can be defined over $\mb Q$ and then it maps divided power $\mb Z$-form into divided power $\mb Z$-form.
So for each $b\in\Phi$, the three families $(\frac{1}{u_{\lambda}}h_{\lambda,b})_{\lambda\in\mc P}$, $(\frac{1}{u_{\lambda}}e_{\lambda,b})_{\lambda\in\mc P}$ and $(\frac{1}{z_{\lambda}}p_{\lambda,b})_{\lambda\in\mc P}$ have the same $\mb Z$-span in $S((\g_{\mb Q}^{\oplus m})^*)$. But then the same holds for the three families $(\frac{1}{u_{\bs\lambda}}h_{\bs\lambda})_{\bs\lambda\in\Theta_\alpha}$, $(\frac{1}{u_{\bs\lambda}}e_{\bs\lambda})_{\bs\lambda\in\Theta_\alpha}$ and $(\frac{1}{z_{\bs\lambda}}p_{\bs\lambda})_{\bs\lambda\in\Theta_\alpha}$. In particular, $\frac{1}{z_{\bs\lambda}}p_{\bs\lambda}$ belongs to $D((\g_{\mb Z}^{\oplus m})^*)$.

Now let $\pi\in S_r$ be of $S_\alpha$ cycle type $\bs\lambda$. Then it is easy to see that
$p_{\bs\lambda}=((X_1,\cdots,X_m)\mapsto f_\pi(X_{\zeta(1)},\ldots,X_{\zeta(r)}))$, $f_\pi$ as in Section~\ref{ss.invs}. So as an element of $S^\alpha(\g_{\mb Q}^{\oplus m})^*$, via the partial polarisation map $P_\alpha$, it is
$$((X_1,\cdots,X_r)\mapsto\sum_{\sigma\in S_\alpha}f_\pi(X_{\sigma(1)},\ldots,X_{\sigma(r)}))=\sum_{\sigma\in S_\alpha}f_{\sigma\pi\sigma^{-1}}\,.$$
So under the $S_r$-equivariant isomorphism $\pi\mapsto f_\pi:kS_r\to((\g^{\otimes r})^*)^G$ the sum of the conjugacy class $[\pi]_{S_\alpha}$ corresponds to divided $p_{\bs\lambda}$.
So the divided $p_{\bs\lambda}$, $\bs\lambda\in\Theta_\alpha$, form a basis of $D^\alpha((\g^{\oplus m})^*)^G=((\g^{\otimes r})^*)^{G\times S_\alpha}$, and the same must then hold for the other two families.

\subsection{Invariants in the algebras $D_s((\g^{\oplus m})^*)$}
We keep the notation of Section~\ref{ss.Young_conjugacy}. Call $\bs\lambda\in\Theta_\alpha$ \emph{$s$-reduced} if $\bs\lambda([j])$ has $<p^s$ ones for all $j\in\{1,\ldots,m\}$. To $\bs\lambda\in\Theta_\alpha$ we can associate its \emph{$s$-reduced form} by repeatedly replacing $p^s$ occurrences of 1 in a $\bs\lambda([j])$ by $p^{s-1}$ occurrences of $p$. We will call two elements of $\Theta_\alpha$ \emph{$s$-equivalent} if they have the same $s$-reduced form. Call two elements of the symmetric group $S_r$ \emph{$(s,\alpha)$-equivalent} if their $S_\alpha$ cycle types are $s$-equivalent.

As in Section~\ref{s.liealg} we can now show that the sums of the $E_\pi$ over the $(s,\alpha)$-equivalence classes belong to $D_s(\g^{\oplus m})^G$, and when $n\ge r$ they form a basis of $D_s^\alpha(\g^{\oplus m})^G$. We only need the lemma in the proof of Theorem~\ref{thm.invs1} for sets $\Lambda$ that are contained in one of the $\Delta_i$. The proof of the theorem below is completely analogous to that of Theorem~\ref{thm.invs2} and we leave this to the reader as well.

\begin{thm}\label{thm.invs3}\ 
\begin{enumerate}[{\rm(i)}]
\item The sums of the divided $p_{\bs\lambda}$'s over the $s$-equivalence classes in $\Theta_\alpha$ belong to $D_s^\alpha((\g^{\oplus m})^*)^G$, and when $n\ge r$ they form a basis of $D_s^\alpha((\g^{\oplus m})^*)^G$.
\item The divided $h_{\bs\lambda}$'s and the divided $e_{\bs\lambda}$'s, both with $\bs\lambda\in\Theta_\alpha$ such that ${\bs\lambda}([j])$ has $<p^s$ ones for all $j\in\{1,\ldots,m\}$, belong to $D_s^\alpha((\g^{\oplus m})^*)^G$, and when $n\ge r$ they form two bases of $D_s^\alpha((\g^{\oplus m})^*)^G$.
\end{enumerate}
\end{thm}

\setcounter{corgl}{0}
\begin{corgl}
The monomials $\prod_{1\le i\le n,b\in \Phi}e_{i,b}^{(m_{i,b})}$, $m_{1,[j]}<p^s$ for $j\in\{1,\ldots,m\}$, belong to $D^r_s((\g^{\oplus m})^*)^G$.
Furthermore, for $r\le n$, those with $\sum_{1\le i\le n,b\in \Phi}m_{i,b}|b|=r$ form a basis of $D^r_s((\g^{\oplus m})^*)^G$.
\end{corgl}
\begin{proof}
Given that $D^r_s((\g^{\oplus m})^*)$ is the direct sum of the $D_s^\alpha((\g^{\oplus m})^*)$, $\alpha\in\mb Z^m$ a composition of $r$, this is just a reformulation of the statement about the $e_\lambda$'s in Theorem~\ref{thm.invs3}.
\end{proof}

\begin{prop}\ \label{prop.inf_invs2}
Assume $r\le n$. Then $D^r_s((\g^{\oplus m})^*)^\g=D^r_s((\g^{\oplus m})^*)^G$.
\end{prop}
\begin{proof}
For $\alpha$ a composition of $r$ we have $D_s^\alpha((\g^{\oplus m})^*)$ is a $G$-submodule of $(\g^{\otimes r})^*$, so this
follows as in the proof of Proposition~\ref{prop.inf_invs1}.
\end{proof}

\begin{corgl}
$$\lim_{\stackrel{\longleftarrow}{n}}D_s((\gl_n^{\oplus m})^*)^{\gl_n}=\lim_{\stackrel{\longleftarrow}{n}}D_s((\gl_n^{\oplus m})^*)^{\GL_n}=D_s((e_{1,[j]})_{1\le j\le m})\ot D((e_{i,b})_{i\text{\,or\,}|b|\ge2}),$$ where $D((e_{i,b})_{i\text{\,or\,}|b|\ge2})$ is graded such that $e_{i,b}^{(t)}$ has degree $ti|b|$, and the limit is in the category of graded $k$-algebras.
\end{corgl}
\begin{proof}
This follows from Proposition~\ref{prop.inf_invs2} and Corollary~1 to Theorem~\ref{thm.invs3}.
\end{proof}

\section{Vectors and covectors}\label{s.vecs_and_covecs}
Let $V=V_n=k^n$ be the natural module for $G$, let $m_1,m_2\ge0$ be integers and put $W=W_n=V^{\oplus m_1}\oplus(V^*)^{\oplus m_2}$. In this section we study the invariants in the algebras $D_s(W^*)$. For $i\in\{1,\ldots,m_1\}$ and $j\in\{1,\ldots,m_2\}$ let $x_i:W\to V$ and $y_j:W\to V^*$ be the $i$-th vector component and $j$-th covector component function and $\la x_i,y_j\ra=((v,w)\mapsto w_j(v_i))\in k[W]^G$ be the bracket function. By Section~\ref{ss.polarisation_Zform} these bracket functions can also be considered as elements of $D(W^*)^G$. The algebra $S(W)$ is $\mb Z^m\times\mb Z^m$-graded and $\mb Z\times\mb Z$-graded and we denote the piece of multidegree $(\alpha^1,\alpha^2)$ by $S^{\alpha^1,\alpha^2}(W)$ and the piece of bidegree $(r_1,r_2)$ by $S^{r_1,r_2}(W)$. We apply analogous notation to the algebras $S(W^*)$, $D(W^*)$ and $D_s(W^*)$.

Let $r_1,r_2\ge0$ be integers and let $\alpha^1=(\alpha^1_1,\ldots,\alpha^1_{m_1})$ and $\alpha^2=(\alpha^2_1,\ldots,\alpha^2_{m_2})$ be compositions of $r_1$ and $r_2$.
As in Section~\ref{ss.Young_conjugacy} we associate to these $\Delta^1_i$, $i\in\{1,\ldots,m_1\}$, $\Delta^2_j$, $j\in\{1,\ldots,m_2\}$, $\zeta_1:\{1,\ldots,r_1\}\to\{1,\ldots,m_1\}$, $\zeta_2:\{1,\ldots,r_2\}\to\{1,\ldots,m_2\}$, and $S_{\alpha^1},S_{\alpha^2}\le S_r$. We have a partial polarisation map
$$P_{\alpha^1,\alpha^2}:S^{\alpha^1,\alpha^2}(W^*)\to S^{\alpha^1,\alpha^2}(W)^*=\Big(\big(V^{\otimes r_1}\otimes(V^*)^{\otimes r_2}\big)^*\Big)^{S_{\alpha^1}\times S_{\alpha^2}}\,.$$
If $F:V^{\oplus r_1}\oplus(V^*)^{\oplus r_2}\to k$ is multilinear and $f$ equals
$$((v_1,\ldots,v_{m_1},w_1,\ldots,w_{m_2})\mapsto F(v_{\zeta_1(1)},\ldots,v_{\zeta_1(r_1)},w_{\zeta_2(1)},\ldots,w_{\zeta_2(r_2)}))\,,$$
then $P_{\alpha^1,\alpha^2}(f)$ equals
$$\big((v_1,\ldots,v_{r_1},w_1,\ldots,w_{r_2})\mapsto\sum_{\sigma\in S_{\alpha^1},\tau\in S_{\alpha^2}}F(v_{\sigma(1)},\ldots,v_{\sigma(r_1)},w_{\tau(1)},\ldots,w_{\tau(r_2)})\big)\,.$$

As in Section~\ref{ss.polarisation_Zform} we obtain isomorphisms $D^{\alpha^1,\alpha^2}(W^*)\cong S^{\alpha^1,\alpha^2}(W)^*$. Under these isomorphisms $D_s^{\alpha^1,\alpha^2}(W^*)$ can be regarded as the multilinear functions $V^{\oplus r_1}\oplus(V^*)^{\oplus r_2}\to k$ which are symmetric in each of the sets of vector positions $\Delta^1_i$ and in each of the sets of covector positions $\Delta^2_i$, and which vanish when the arguments in $p^s$ positions within a $\Delta^\iota_i$, $\iota\in\{1,2\}$, are the same. Furthermore, these isomorphisms are compatible with the isomorphism $D(W^*)\cong S(W)^{*\rm gr}$ from Section~\ref{ss.polarisation_Zform}.

Assume now that $\alpha^1$ and $\alpha^2$ above are compositions of $r$.
The group $S_r\times S_r$ acts on $S_r$ via $(\sigma,\tau)\cdot\pi=\sigma\pi\tau^{-1}$. Each $S_{\alpha^1}\times S_{\alpha^2}$-orbit has a unique representant $\pi$ such that $\pi$ is increasing on each $\Delta^2_j$ and $\pi^{-1}$ is increasing on each $\Delta^1_i$. Let $\pi\in S_r$. Put $\Delta^1_{ij}=\Delta^1_i\cap\pi(\Delta^2_j)$ and $m_{ij}=|\Delta^1_{ij}|$ for $1\le i\le m_1, 1\le j\le m_2$. Then
\begin{equation}\label{eq.mij}
\alpha^1_i=\sum_{j=1}^{m_2}m_{ij}\text{\ and\ }\alpha^2_j=\sum_{i=1}^{m_1}m_{ij}\,.
\end{equation}
For $\sigma,\tau\in S_r$ we have $(\sigma,\tau)\in S_{\alpha^1}\times S_{\alpha^2}$ and $\sigma\pi\tau^{-1}=\pi$ if and only if $\sigma\in S_{\alpha^1}\cap\pi S_{\alpha^2}\pi^{-1}$ and $\tau=\pi^{-1}\sigma\pi$. So the $S_{\alpha^1}\times S_{\alpha^2}$-centraliser of $\pi$ has size $|S_{\alpha^1}\cap\pi S_{\alpha^2}\pi^{-1}|=\prod_{1\le i\le m_1, 1\le j\le m_2}m_{ij}!$.

Conversely, if we are given integers $m_{i,j}\ge0$, $1\le i\le m_1, 1\le j\le m_2$, which sum to $r$, then we can define $\alpha^1$ and $\alpha^2$ by \eqref{eq.mij} and we can define the $\Delta^1_i$ and $\Delta^2_j$ as before. We divide each $\Delta^1_i$ into $m_2$ consecutive intervals $\Delta^1_{i1},\ldots,\Delta^1_{im_2}$ and we divide each $\Delta^2_j$ into $m_1$ consecutive intervals $\Delta^2_{1j},\ldots,\Delta^2_{m_1j}$ such that $\Delta^1_{ij}$ and $\Delta^2_{ij}$ have length $m_{ij}$.
Now we define $\pi\in S_r$ by requiring that $\pi:\Delta^2_{ij}\to\Delta^1_{ij}$ is increasing. Then $\pi$ is increasing on each $\Delta^2_j$ and $\pi^{-1}$ is increasing on each $\Delta^1_i$.

\begin{prop} Let $r_1,r_2\ge0$ be integers.\label{prop.invs}
\begin{enumerate}[{\rm(i)}]
\item If $r_1\ne r_2$, then $D^{r_1,r_2}(W^*)=0$. If $r_1=r_2=r$, then the divided power monomials in the $\la x_i,y_j\ra$ of bidegree $(r,r)$ belong to $D_1^{r,r}(W^*)^G$, and when $n\ge r$ they form a basis of $D^{r,r}(W^*)^G=D_1^{r,r}(W^*)^G$.
\item If $n\ge r_1,r_2$, then $D^{r_1,r_2}(W^*)^\g=D^{r_1,r_2}(W^*)^G$.
\end{enumerate}
\end{prop}
\begin{proof}
(i).\ By considering the action of the centre of $G$ it follows that if $r_1\ne r_2$, then $D^{r_1,r_2}(W^*)^G=0$, so we assume now that $r_1=r_2=r$.
By Lemma~\ref{lem.divpower} the given monomials belong to $D_1^{r,r}(W^*)$.
Denote the vector and covector component functions of $V^{\oplus r}\oplus(V^*)^{\oplus r}$ by $\ov x_i$ and $\ov y_i$, $i\in\{1,\ldots,r\}$.
The function $f_\pi\in(\g^{\otimes r})^*$ from Section~\ref{ss.invs} can also be seen as an element of $(V^{\otimes r}\otimes(V^*)^{\otimes r})^*$. Then we have
$f_\pi=\prod_{i=1}^r\la\ov x_{\pi(i)},\ov y_i\ra$ and we see that the map $\pi\mapsto f_\pi$ is $S_r\times S_r$-equivariant.

Let $m_{i,j}\ge0$, $1\le i\le m_1, 1\le j\le m_2$, be integers which sum to $r$. Define $\alpha^1$ and $\alpha^2$ by \eqref{eq.mij} and then define $\Delta^1_i$, $\Delta^2_j$, $\zeta_1$, $\zeta_2$, $S_{\alpha^1}$, $S_{\alpha^2}$ as in Section~\ref{ss.partial_polarisation}, and define $\pi$ as before the proposition.
It is easy to see that $\prod_{1\le i\le m_1,1\le j\le m_2}\la x_i,y_j\ra^{m_{ij}}=\prod_{i=1}^r\la x_{\zeta_1(\pi(i))},y_{\zeta_2(i)}\ra$.
So as an element of $S^{r,r}(W_{\mb Q})^*$, via the partial polarisation map $P_{\alpha^1,\alpha^2}$, it is
$\sum_{\sigma\in S_{\alpha^1},\tau\in S_{\alpha^2}}\prod_{i=1}^r\la\ov x_{\sigma(\pi(i))},\ov y_{\tau(i)}\ra=\sum_{\sigma\in S_{\alpha^1},\tau\in S_{\alpha^2}}f_{\sigma\pi\tau^{-1}}$. So under the $S_r\times S_r$-equivariant isomorphism $\pi\mapsto f_\pi:kS_r\to\big((V^{\otimes r}\otimes(V^*)^{\otimes r})^*\big)^G$ the sum of the orbit $[\pi]_{S_{\alpha^1}\times S_{\alpha^2}}$ corresponds to
$\prod_{1\le i\le m_1,1\le j\le m_2}\la x_i,y_j\ra^{(m_{ij})}$. So these divided power monomials form a basis of $D^{r,r}(W^*)^G=\bigoplus_{\alpha^1,\alpha^2}D^{\alpha^1,\alpha^2}(W^*)^G$.\\
(ii).\ As $D^{\alpha^1,\alpha^2}(W)^*=\Big(\big(V^{\otimes r_1}\otimes(V^*)^{\otimes r_2}\big)^*\Big)^{S_{\alpha^1}\times S_{\alpha^2}}$, this follows from Lemma~\ref{lem.inf_invs}(ii).
\end{proof}

Note that we have a natural embedding $V_{n-1}\hookrightarrow V_n$ by adding a zero component in the $n$-th position, and a natural embedding $V_{n-1}^*\hookrightarrow V_n^*$ by extending a function $f\in V_{n-1}^*$ by sending the $n$-th standard basis vector to 0. This gives us a natural embedding $W_{n-1}\hookrightarrow W_n$, and we get restriction maps for the algebras $(k[W_n])_{n\ge1}$, $(D(W_n^*))_{n\ge1}$ and $(D_s(W_n^*))_{n\ge1}$. From the previous proposition we immediately obtain the following corollary, where we may omit the subscript $s$.
\begin{cornn}
$$\lim_{\stackrel{\longleftarrow}{n}}(D_s(W_n^*))^{\gl_n}=\lim_{\stackrel{\longleftarrow}{n}}(D_s(W_n^*))^{\GL_n}=D(\la x_i,y_j\ra_{1\le i\le m_1,1\le j\le m_2}),$$ where the grading is such that $\la x_i,y_j\ra^{(t)}$ has degree $2t$, and the limit is in the category of graded $k$-algebras.
\end{cornn}

\begin{rems}\label{rems.vecs_and_covecs}
1.\ It is immediate from classical invariant theory, see \cite{DeCProc}, that the algebras $(k[W_n])_{n\ge1}$ have the restriction property.\\
2.\ Since $W_n\cong(V_n^{\oplus m_2}\oplus (V_n^*)^{\oplus m_1})^*$, we get restriction maps $W_n\to W_{n-1}$. From the description of $\bigwedge(W_n)^G$ in \cite[Sect~5]{AR} it is clear that the algebras $\bigwedge(W_n)_{n\ge1}$ have the restriction property. This implies that when $p=2$, the algebras $(A_1(W_n))_{n\ge1}$ have the restriction property.\\
3.\ For $p=3$ the algebras $(D(W^*_n))_{n\ge1}$ and $(D_s(W^*_n))_{n\ge1}$ don't have the restriction property. I checked with the computer for $p=3,n=2,m_1=1$,$m_2=3$ that $\dim D^r_1(W^*_n)=1, 0, 3, 0, 6, 0, 11, 0, 15$ for $r=0,\ldots,8$ and $0$ for $r>8$, and that the dimensions of the span of the invariants from Proposition~\ref{prop.invs} in degrees $=0,\ldots,8$ are $1, 0, 3, 0, 6, 0, 10, 0, 15$. In degree 6 the invariant $x_{1}x_{2}(x_{1}y_{21} - x_{2}y_{22}) (y_{12}y_{31}- y_{11}y_{32})$ is outside this span, where $y_{ji}$ denotes the $i$-th component of the $j$-th covector.\\
4.\ Similar to \cite[Sect~5]{AR} one could try to determine the invariants in $A_1(W_n)=D_1(W_n^*)$ by using the isomorphism $A_1(W_n)\cong A_1((V_n^*)^{\oplus m})\otimes\det^{m_1(1-p)}$, $m=m_1+m_2$, of $\GL_n$-modules, and then use the commuting $\GL_m$-action. Let $U_n\le\GL_n$ be the subgroup of upper uni-triangular matrices. Then we get $A_1(W_n)^{\GL_n}\cong A_1((V_n^*)^{\oplus m})^{U_n}_{m_1(p-1)\un 1_n}$, where $\un 1_n$ is the all-one vector of length $n$. Now one could hope that
$A_1((V_n^*)^{\oplus m})^{U_n}_{(p-1)\nu}\cong\Delta_{\GL_m}((p-1)\nu^T)$, $\Delta_{\GL_m}(\mu)$ the Weyl module of highest weight $\mu$ and $\nu^T$ the transpose of $\nu$, at least for $\nu$ a multiple of $\un 1_n$. Indeed the analogue for the exterior algebra holds by \cite{ABW} or \cite{AR}. 
However, in the case $p=3,n=2,m_1=1,m_2=3$, $A_1((V_2^*)^{\oplus 4})^{U_2}_{(2,2)}$ is not even a quotient of some Weyl module. Indeed its socle and ascending radical series both have two layers: the first one is the irreducible $L_{\GL_4}(2,2,0,0)$ of dimension 19 and the second layer is $L_{\GL_4}(1,1,1,1)\oplus L_{\GL_4}(4,0,0,0)$ of dimension $1+16=17$. The Weyl module $\Delta_{\GL_4}(4,0,0,0)$ has dimension $35$ and the two layers of its socle and ascending radical series are $L_{\GL_4}(2,2,0,0)$ and $L_{\GL_4}(4,0,0,0)$.
\end{rems}


\begin{thebibliography}{99}
\bibitem{AR} A.~M.~Adamovich, G.~L.~Rybnikov, {\it Tilting modules for classical groups and Howe duality in positive characteristic}, Transform. Groups {\bf 1} (1996), no. 1-2, 1-34.
\bibitem{ABW} K.~Akin, D.~A.~Buchsbaum, J.~Weyman, {\it Schur functors and Schur complexes}, Adv. in Math. {\bf 44} (1982), no. 3, 207-278.
\bibitem{Bou} N.\ Bourbaki, {\it Alg\`ebre}, Chaps. 1, 2 et 3, Hermann, Paris, 1970.
\bibitem{DeCProc} C.\ De Concini, C.\ Procesi, {\it A characteristic free approach to invariant theory}, Advances in Math. {\bf 21} (1976), no. 3, 330-354.
\bibitem{Don} S.~Donkin, {\it Invariant functions on matrices}, Math. Proc. Cambridge Philos. Soc. {\bf113} (1993), no. 1, 23-43.
\bibitem{FP} E.~M.~Friedlander and B.~J.~Parshall, {\it Rational actions associated to the adjoint representation}, Ann. Sci. \'Ecole Norm. Sup. (4) {\bf20} (1987), no. 2, 215-226.
\bibitem{Green} J.~A.~Green, {\it Polynomial representations of ${\rm GL}\sb{n}$}, Lecture Notes in Mathematics, {\bf 830}, Springer-Verlag, Berlin-New York, 1980.
\bibitem{Hab} W.~J.~Haboush, {\it Central differential operators on split semisimple groups over fields of positive characteristic}, S\'eminaire d'Alg\`ebre Paul Dubreil et Marie-Paule Malliavin, 32\`eme ann\'ee (Paris, 1979), pp. 35-85, Lecture Notes in Math. {\bf795}, Springer, Berlin, 1980.
\bibitem{Jan} J.\ C.\ Jantzen, {\it Representations of algebraic groups}, Pure and Applied Math., vol.~131. Academic Press, Boston, 1987.
\bibitem{Mac} I. G. Macdonald, {\it Symmetric functions and Hall polynomials}, Second edition, Oxford University Press, New York, 1995.
\bibitem{OO} A.~Yu.~Okounkov and G.~I.~Olshanskii, {\it Shifted Schur functions} (Russian), Algebra i Analiz {\bf 9} (1997), no.2, 73-146; translation in
St. Petersburg Math. J. {\bf 9} (1998), no.2, 239-300.
\bibitem{Pr} A.~Premet, {\it Special transverse slices and their enveloping algebras}, Adv. Math. {\bf170} (2002), no. 1, 1–55.
\bibitem{PrT} A.\ A.\ Premet and R.\ H.\ Tange, {\it Zassenhaus varieties of general linear Lie algebras}, J. Algebra {\bf 294} (2005), no. 1, 177-195.
\bibitem{Skr} S.~Skryabin, {\it Invariants of finite group schemes}, J. London Math. Soc. (2) {\bf65} (2002), no. 2, 339-360.
\bibitem{T} R.~Tange, {\it On the first restricted cohomology of a reductive Lie algebra and its Borel subalgebras}, Ann. Inst. Fourier (Grenoble) {\bf69} (2019), no. 3, 1295-1308.
\end{thebibliography}
\end{document}